\documentclass{amsart}
\usepackage{tikz, float}
\usepackage[margin=1in]{geometry}

\newcommand{\MIN}{\mathcal{MIN}}
\newcommand{\PMF}{\mathcal{PMF}}
\newcommand{\MF}{\mathcal{MF}}
\newcommand{\UE}{\mathcal{UE}}

\newcommand{\Star}{{\sf S}}
\newcommand{\ZZ}{{\sf Z}}
\newcommand{\C}{\ensuremath{\mathbb{C}}}

\newcommand{\R}{\ensuremath{\mathbb{R}}}
\renewcommand{\H}{\ensuremath{\mathbb{H}}}
\newcommand{\Z}{\ensuremath{\mathbb{Z}}}

\newcommand{\T}{\mathcal{T}}

\newcommand{\Teich}{~Teich\-m\"uller~}
\newcommand{\binf}{\partial_\infty}
\newcommand{\bhor}{\partial_h}
\newcommand{\Thin}{\mathop{\rm Thin}}
\newcommand{\Prod}{\mathop{\rm Prod}}
\newcommand{\cc}{{\sf c}}
\newcommand{\sdot}{\! \cdot \!}
\newcommand{\ds}{d_\star}
\newcommand{\dc}{d_\mathcal{C}}

\newcommand*{\mybox}[1]{%
  \framebox{\raisebox{0pt}[0.5\baselineskip][0.05\baselineskip]{%
    #1}}}

\theoremstyle{remark}

\theoremstyle{definition}

\newtheorem{theorem}{Theorem}
\newtheorem*{theorem*}{Theorem}
\newtheorem*{thma}{Theorem A}
\newtheorem*{corb}{Corollary B}
\newtheorem*{lemc}{Lemma C}
\newtheorem*{thmd}{Theorem D}
\newtheorem*{thme}{Theorem E}
\newtheorem*{conjecture}{Conjecture}

\newtheorem{proposition}[theorem]{Proposition}
\newtheorem{lemma}[theorem]{Lemma}

\begin{document}

\title{Stars at infinity in Teichm\"uller space}
\author{Moon Duchin and Nathan Fisher}
\date{\today}

\begin{abstract}
We investigate a metric structure on the Thurston boundary of \Teich space.  To do this, we develop tools in sup metrics and apply Minsky's theorem.
\end{abstract}

\maketitle

\section{Introduction}

In geometric topology and geometric group theory, the study of boundaries at infinity has repeatedly provided crucial tools in establishing dynamical results, often 
with surprising geometric and algebraic consequences.  Negatively curved spaces enjoy an extremely rich boundary theory through Gromov's {\em visual boundary},
constructed through geodesics or quasigeodesics.  Unfortunately, visual boundaries are much less well-behaved outside of the negative curvature setting,
even in nonpositively curved spaces.  Still, boundaries for metric spaces sometimes exhibit points called {\em hyperbolic points} with some of the properties
enjoyed by boundary points in hyperbolic spaces, such as so-called ``visibility" properties that describe which pairs of boundary points can be connected by geodesics.  
For any bordification of a metric space, a construction of Anders Karlsson induces extra structure on the boundary
by breaking it down into sets called {\em stars} \cite{karlsson-stars}. The star of a boundary point $\xi$ is built through a purely metric definition, and contains other boundary points that are in a
certain sense metrically indistinguishable:  as we will define below, the star of $\xi$ consists of those boundary points which cannot be separated from $\xi$ by halfspaces.  When the star is a single point, the direction has hyperbolic features.  

\Teich space has been a nexus of attention in topology and group theory.  The \Teich space $\T(S)$ parametrizes geometric structures on the topological surface $S$,
and it carries its own geometries which are well-studied in their own right, especially the \Teich metric, a complete geodesic Finsler metric with a very nice interpretation
in terms of flat structures on $S$.
Much attention has been paid to the Thurston boundary 
$\PMF(S)$ of the \Teich space $\T(S)$ for a surface $S=S_{g,n}$ of finite type.  
Its elements, called {\em projective measured foliations}, can be regarded as completing  the simple closed curves on $S$ with respect to a natural topology.  
Where $h=6g-6+2n$, Thurston established that this boundary is a sphere of dimension
$h-1$ compactifying $T(S)$, which itself is topologically a ball of dimension $h$.  Unfortunately, the boundary lacks some of the nice 
geometric properties present in boundaries of hyperbolic spaces:  it is basepoint-dependent \cite{Kerckhoff}, it has limited visibility
(not every two points on the boundary can be connected by a geodesic) \cite{Gardiner-Masur}, 
and indeed some \Teich geodesics have large accumulation sets on the boundary \cite{lenzhen}, while others are not approached by geodesics \cite{masur-thesis}.  Karlsson asked how halfspaces in the \Teich metric separate points in the Thurston boundary $\PMF$, and he predicted that it lines up with intersection of foliations \cite{karlsson-stars}. That is, Karlsson conjectured that disjoint foliations are precisely the ones that can't be separated.
\begin{conjecture}[Karlsson] $\Star(F)=\ZZ(F)$ for all $F\in\PMF$. \end{conjecture}

He proved part of that in his original paper.

\begin{theorem*}[Karlsson, \cite{karlsson-stars} Thm 44] 
For minimal foliations, only those with the same underlying topology can belong to a given star:
$\Star(F)\subseteq \ZZ(F)$ for all $F \in \MIN$.
\end{theorem*}

In this paper, we study the stars in the Thurston boundary, addressing Karlsson's conjecture.
We show that the hyperbolic points in the boundary are precisely the {\em uniquely ergodic foliations}, a well-studied subset of 
$\PMF$ that coincides with its Poisson boundary, as established by Kaimanovich--Masur \cite{kaim-mas}.

\begin{thma} [Stars are bigger than zero-sets]
Non-intersecting foliations always belong to the same star:
$\ZZ(F)\subseteq \Star(F)$ for all $F\in\PMF$. \label{DZ}\end{thma}

\begin{corb} $\Star(F)=\{F\} \iff F\in\UE$. \end{corb}

Furthermore, the star structure on $\PMF$ induces a metric on the set of simple closed curves $\mathcal S(S)$. We develop tools in the hope of relating this star metric to the metric coming from the curve graph $\mathcal{C}(S)$.

\begin{conjecture}
The star metric $\ds$ and the curve complex distance $\dc$ are isometric on the set of simple closed curves $\mathcal{S}\subset \PMF$.
\end{conjecture}
This would tell us that the distance function induced by disjointness is the same as the distance function induced by star-membership:
you can use \Teich distance alone to see a copy of the curve complex in the boundary.

The main ingredient in the proof of Theorem A is Minsky's theorem that the regions of \Teich space where some curves are very short (the ``thin parts")
are additively well-approximated by sup metrics.  We will deduce the desired conclusions from establishing that stars
in sup metrics are suitably large.  

Along the way we establish three results that may be of independent interest for the study of boundaries and random walks.
Lemma C gives a necessary and sufficient condition for $\eta\in\Star(\xi)$ in terms of sequences and a metric inequality.
For vector spaces with sup metrics, 
Theorem D describes the horofunctions explicitly in terms of a family of geodesics and derives a topology on the boundary; Theorem E constructs the stars in that horoboundary.

%


\subsection*{Acknowledgements} Thanks to Joseph Maher for collaborating on an earlier incarnation 
of this project. Thanks to Sunrose Shrestha, Thomas Weighill, Chris Leininger, Howie Masur, Ruth Charney, and Kasra Rafi for their ears and insights,  and thanks to Anders Karlsson for suggesting the problem and for many interesting and useful conversations.

\section{Background and basic properties}

Recall Thurston's main tools to understand \Teich geometry on $S$, the simple closed curves $\mathcal{S}$ and the measured foliations $\MF$, 
which can be related with an intersection form $i(\cdot,\cdot)$.  
In the intervening years, the combinatorics of the set of simple closed curves $\mathcal S(S)$ has been mined very productively:  the {\em curve complex} 
declares two curves to be adjacent if they are disjointly realizable, and we define a distance function as the length metric on the graph.  
In the 1980s, Bonahon showed how to interpret $i(\cdot,\cdot)$ as a continuous bilinear form on the larger space of geodesic currents, and how to view $\T(S)$ and
$\PMF(S)$ as sitting compatibly inside  currents \cite{bon-currents}, with the Thurston compactification showing up as directly analogous to the 
sphere at infinity for a hyperboloid in Lorentz geometry.  Closed curves and measured foliations also embed in the  space of currents.
Let us define the {\em zero-set} of a foliation to be all the foliations that miss it:
$$\ZZ(F):= \{ G\in\PMF : i(F,G)=0\}.$$
We can designate sub-classes of foliations 
 $$\UE \subset \MIN \subset \PMF,$$
where 
{\em minimal foliations} $(\MIN)$ are those for which every leaf is topologically dense in $S$, and 
{\em uniquely ergodic foliations} $(\UE)$ are minimal foliations that carry a unique transverse measure.
Then, since every transverse measure is a convex combination of finitely many mutually singular ergodic measures,
these zero-sets can be thought of as polyhedra in the boundary, and 
we have $$\ZZ(F)=\{F\} \iff F\in \UE.$$

Next, we review the general theory of halfspaces and stars at infinity developed by Karlsson in \cite{karlsson-stars}. 
The definitions serve in a more general setting, but here we assume $X$ is a complete proper geodesic metric space and fix a basepoint $x_0 \in X$. For a subset $W \subset X$ and a constant $C\geq0$, define the halfspace $H(W, C)$ by
$$H(W,C) = H^{x_0}(W,C) := \{z : d(z,W) \leq d(z,x_0) + C\}.$$
Note that if $W$ is equal to a point $y$ and $C = 0$, then $H(\{y\}, 0)$ defines a standard halfspace.

We let $\overline X$ be any Hausdorff bordification of $X$ and denote the boundary $\partial X = \overline X - X$.
Two examples of bordifications frequently seen in the setting of geometric group theory include the visual boundary $\binf(X)$, the set equivalence classes of basepointed geodesic rays, and the horofunction boundary $\bhor(X)$ obtained by embedding $X$ into the space $C(X)$ of real-valued continuous functions on $X$. To be precise, fix a point $x_0$ in $X$, and embed $X$ into $C(X)$ via the map
$$
\Psi: z \mapsto d(z, \cdot) - d(z, x_0).
$$ 
Define the horofunction boundary to be $\bhor(X) = \overline{\Psi(X)} \backslash \Psi(X).$ A sequence $x_n$ in $X$ converges to a point in the boundary if and only if the sequence of functions $d(x_n, \cdot) - d(x_n, x_0)$ converges uniformly on compact sets to a function not in the image of $\Psi$.

To relate the visual boundary to the horofunction boundary, when $\gamma$ is a geodesic ray based at $x_0$ in $X$, then it is a simple exercise to show that $\Psi(\gamma(t))$ converges as $t \to \infty$ to a horofunction in the boundary. This class of horofunctions coming from geodesic rays are called {\em Busemann functions}.

Returning to the general setting, we assume $\overline X$ is a Hausdorff bordification, and we let $\mathcal{V}_\xi$ denote the collection of open neighborhoods in $\overline X$ of a boundary point $\xi \in \partial X$. The {\em star based at $x_0$} of $\xi$ is
$$
S^{x_0}(\xi) := \bigcap\limits_{V \in \mathcal{V}_\xi} \overline{H(V,0)},
$$
where closures are taken in $\overline X$. A priori, this definition could depend on the basepoint $x_0$, so the {\em star} of $\xi$ is defined to be
$$
S(\xi) := \overline{\bigcup\limits_{C\geq 0}\bigcap\limits_{V \in \mathcal{V}_\xi} \overline{H(V,C)}}.
$$

The combinatorially defined {\em star-distance} on $\partial X$ is given by setting $d_\star$ to
be the maximal distance function satisfying
\begin{align*}
d_\star(\xi,\eta)=0 &\iff \xi = \eta,\\
d_\star(\xi,\eta)=1& \iff \eta\in \Star(\xi) ~\hbox{or}~ \xi\in \Star(\eta).
\end{align*}

{\bf Examples}
\begin{enumerate}
\item If $X=\H^n$, then
$\Star(\xi)=\{\xi\}$ for all $\xi\in\binf(X)\cong \bhor(X) \cong S^{n-1}$.  Let us call points whose star is a singleton the {\em star hyperbolic points} of the boundary.
In this case, as in every boundary with all hyperbolic points, the star-diameter is infinite.
\item The stars in $\binf(\R^n)\cong \bhor(\R^n) \cong S^{n-1}$ are closed hemispheres centered at $\xi$.
Here the star-diameter is two.
\item As Karlsson shows, the stars in CAT(0) spaces are the balls of radius $\pi/2$ in the Tits angular metric on $\partial_\infty X$.
(Note this generalizes both of the previous examples.)
\end{enumerate}

The generality of the construction of stars makes them a powerful tool; stars are defined in such a way as to make them basepoint independent for any Hausdorff bordification of a complete metric space. The cost of this generality is difficulty deriving properties.
For instance, symmetry of star-membership has been an open question until recently.
In a new preprint \cite{jones2020asymmetry}, Jones and Kelsey settle the question negatively by exhibiting points $\alpha$ and $\beta$ in the horofunction boundary of a Diestel-Leader graph such that $\beta \in \Star(\alpha)$, but $\alpha \not\in \Star(\beta)$. Jones--Kelsey also give a sufficient condition for star membership.  
We strengthen that to a necessary and sufficient condition as follows.

\begin{lemc}[Sequence criterion] Suppose $\overline X = X \cup \partial X$ is compact and first countable. Then $\eta\in \Star(\xi)$ if and only if for
every neighborhood $U$ of $\eta$ in $\overline X$, there are sequences
$x_n\to \xi$, $y_n\to U$ and a constant $C\ge 0$ such that
\begin{equation} d(y_n,x_n)\le d(y_n,x_0) + C.   \label{dist-ineq} \tag{$\star$}
\end{equation}

In particular, if there exist $C\ge 0$ and  sequences $x_n\to \xi$ and $y_n\to\eta$
as in \eqref{dist-ineq},
then $\eta\in \Star(\xi)$.
\end{lemc}

\begin{proof}
Let $W_k$ be a neighborhood basis of $\xi$.
Then $\eta\in \Star(\xi)$ is equivalent to the existence of a sequence of
points in the boundary $\eta_i\to\eta$ and associated constants $C_i$
so that
$$\eta_i \in \bigcap_k \overline{H^{x_0}(W_k,C_i)},$$  So for all $k$, $\eta_i \in
\overline{H^{x_0}(W_k,C_i)}$, which means that for each $k$ there is a sequence
$y_{i,n,k} \in H^{x_0}(W_k,C_i)$ with $y_{i,n,k} \to \eta_i$ as $n\to\infty$.
But that means that for every $k$ there is a point $x_{i,n,k}\in W_k$
such that
$$d(y_{i,n,k}\ , \  x_{i,n,k}) \le d(y_{i,n,k}\ , \ x_0) + C_i.$$
Given a neighborhood $U$ of $\eta$, choose and fix a sufficiently
large value of $i$ so that $\eta_i\in U$.
Then let $x_n=x_{i,n,n}$ and $y_n=y_{i,n,n}$.  By construction,
$x_n\to \xi$, $y_n\to U$, and
$d(y_n,x_n)\le d(y_n,x_0) + C$ for $C=C_i$.

For the other direction, let us assume sequences exist as above and prove that $\eta\in \Star(\xi)$. 
Again, let $W_k$ be a neighborhood basis at $\xi$ and let $U_i$ be a neighborhood basis at $\eta$. 
The hypothesis says that for every $i$, there are sequences $y_{n,i}\to U_i$, $x_{n,i}\to \xi$, and there is a constant
$C_i$ such that $y_{n,i}\in H^{x_0}(x_{n,i},C_i)$ for all $n$. Observe that for any fixed $i$, as $n\to \infty$, $y_{n,i}$ must go to infinity, i.e., leave all bounded sets in $X$. Indeed, if $y_{n,i}$ were bounded, the distance $d(y_{n,i}, x_0)$ would be bounded above by a number $M_i>0$, and, therefore, $d(y_{n,i}, x_{n,i})$ would also be bounded above by $M_i+C_i$. We know, however, that as $n\to\infty$,  the sequence $x_{n,i} \to \xi \in \partial X$, resulting in a contradiction.

Since $\overline X$ is compact, for a fixed $i$, the sequence $y_{n,i}$ must have a limit point in $\overline{U_i}$. In particular, since the sequence leaves all bounded sets, $y_{n,i}$ has a subsequence that converges to a point in $\partial X \cap U_i$. Pass to this subsequence, and call this limit point $\eta_i$. For any $k$ and for $n$ sufficiently large, we have $x_{n,i} \in W_k$. So for $n$ large enough, $y_{n,i} \in H^{x_0}(x_{n,i}, C_i)$ implies that $y_{n,i} \in H^{x_0}(W_k, C_i)$. Therefore we have $\eta_i \in \overline{H^{x_0}(W_k, C_i)}$ for all $k$. This is true for arbitrary $U_i$, so we can find a sequence of points $\eta_i \in \partial X$ such that
\[
\{\eta_i\} \subseteq \bigcup\limits_{C\geq 0}\bigcap\limits_{k} \overline{H^{x_0}(W_k,C)}.
\]

There is a subsequence of $\{\eta_i\}$  which converges to $\eta$ by construction, and so we have
\[\eta \in \Star(\xi) = \overline{\bigcup\limits_{C\geq 0}\bigcap\limits_{k} \overline{H^{x_0}(W_k,C)}}.  \qedhere \] 

\end{proof}

\begin{lemma}[Semicontinuity of stars]
Suppose the bordification $\overline X = X \cup \partial X$ is a metrizable space. If $\eta_n\in \Star(\xi_n)$ and $\eta_n\to\eta$, $\xi_n\to\xi$, then $\eta\in \Star(\xi)$.
\label{semicont}
\end{lemma}

\begin{proof}
First, pass to a subsequence of $\eta_n$, relabeling indices as needed, such that $\eta_n \in B_{1/n}(\eta)$ for all $n>0$. Let $U_k^n$ be a neighborhood basis for $\eta_n$. In particular, take $U^n_k$ to be the metric balls $B_{1/k}(\eta_n)$.  Then since $\eta_n\in \Star(\xi_n)$, 
we know by Lemma C that for fixed $n$ and any $k$, there is a $C_k\ge 0$ and sequences $x_{i,n,k}$ and
$y_{i,n,k}$, tending to $\xi_n$ and $U_k^n$ respectively as $i\to\infty$, satisfying $$d(y_{i,n,k}\ , \  x_{i,n,k}) \le d(y_{i,n,k}\ , \ x_0) + C_k.$$

For every neighborhood $U$ of $\eta$, there is a radius $R>0$ such that $B_R(\eta) \subseteq U$. Note that by the triangle inequality $U^n_k = B_{\frac1k}(\eta_n) \subseteq B_{\frac1n + \frac1k}(\eta)$. Therefore, there exist constants $N$ and $K$ satisfying $\frac1N + \frac1K < R$ for which $n\geq N, k\geq K \implies U_k^n \subset U$.  Now let
$x_n=x_{n,n,K}$ and $y_n=y_{n,n,K}$.  We have $x_n\to \xi$ and $y_n\to U$, while $d(y_n,x_n)\le d(y_n,x_0)+C$ for $C=C_K$. 
\end{proof}

\section{Stars in sup metrics}
In this section, we will take a digression into the geometry of sup metrics, which appear in the thin parts of \Teich space, to be described below. In particular, we will study the horofunction boundary of $X=(\R^2,\sup).$
In doing so, we will recover the results of Gutierrez \cite{gutierrezhorofunction} on the horofunction boundary, but through a more geometric proof. We will then explore the stars in this sup metric, which are of independent interest but also provide intuition for the proof of the main theorem on stars in \Teich space.

\subsection{Directional sequences and geodesics}
Here, we will use the term {\em geodesic ray}  in a metric space $X$ to refer either to isometric  
maps $[0,\infty) \to X$
or to the image of such a map in the space $X$, which are characterized by the betweenness relation:  for any three points $p,q,r$ in order along the curve,
$d(p,q)+d(q,r)=d(p,r)$. 
Since the sup metric on $\R^2$ is translation-invariant, to understand all geodesics, it suffices to describe geodesic rays based at the origin. 
It quickly follows from the betweenness description that sup geodesics are parametrized by arclength by either their $x$ coordinate, $y$ coordinate, or the negative of one of these.  
Accordingly,  a  (unit-speed) {\em northerly geodesic}  can be written in the form $\left( x(t),\ t\right)$,
and similarly for easterly, southerly, and westerly, which covers all possibilities.

We can define a {\em northerly sequence} $\{z_n= (x_n, y_n)\}$  of points in $\R^2$ by 
the property that $y_n>y_{n-1}$ and 
$$y_{n+1} -y_n \geq |x_{n+1}-x_n|,$$
or in other words, the  the northward displacement dominates the east and west displacement.
(See Figure~\ref{fig:northerly}.)
Note that northerly geodesics are characterized by all exiting sequences being northerly sequences, and that every northerly sequence can be interpolated to a northerly geodesic.
Calling a curve of one of these four types {\em directional}, we have observed that {\em every sup geodesic is directional}.

\begin{figure}[ht]
\begin{tikzpicture}[scale=.5]
\draw [<->] (-5,0)--(5,0);
\draw [<->] (0,-2)--(0,6);
\filldraw (0,0) circle (0.1) node[below right] {$z_0$};
\draw [opacity=.5] (-5,5)--(0,0)--(5,5);
\filldraw (1.3,2) circle (0.1) node[below left] {$z_1$};
\draw [opacity=.5] (-2.2,5.5)--(1.3,2)--(4.8,5.5);
\filldraw (-.2,3.8) circle (0.1) node[below left] {$z_2$};
\draw [opacity=.5] (-1.9,5.5)--(-.2,3.8)--(1.5,5.5);

\begin{scope}[xshift=14cm]
\draw [<->] (-5,0)--(5,0);
\draw [<->] (0,-2)--(0,6);
\draw [->,green,ultra thick,opacity=.5] (0,0) -- (0,2)--(5,7);
\draw [->,blue, ultra thick,opacity=.5] (0,0) -- (0,3)--(-4,7);

\end{scope}
\end{tikzpicture}
\caption{To the left, a northerly sequence.  Such a sequence can be linearly interpolated to form a northerly geodesic.  To the right, the northerly geodesics $\alpha^{\rm NE}_2$
and $\alpha^{\rm NW}_3$.}\label{fig:northerly}
\end{figure}
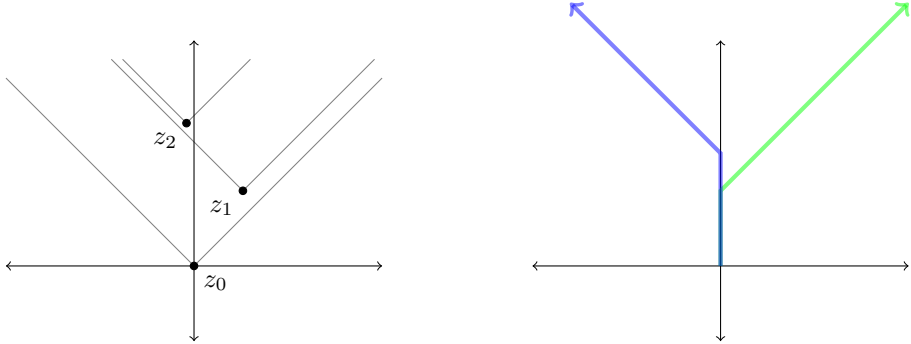

We now define several families of geodesics that we will use below.
For $m\ge 0$, let $$\alpha^{\rm NE}_m(t)=(\max(0,t-m),t).$$  These first travel north  for $m$ units, then diagonally northeast.  There are eight families defined similarly, each starting 
in an axial direction (N/S/E/W) for $m\ge 0$ units, then turning $\pi/4$ and continuing.  
(So for instance, $\alpha^{\rm EN}_3$ is an easterly geodesic, and $\alpha^{\rm EN}_0=\alpha^{\rm NE}_0$ is both easterly and northerly.)
Besides these, we also define the pure axial geodesics $\alpha^{\rm N}$, $\alpha^{\rm S}$, $\alpha^{\rm E}$, $\alpha^{\rm W}$, which can be thought of as $m\to\infty$ limits from the other families.

%

\subsection{Boundaries of sup metrics}
Recall that the visual boundary
is the set of geodesic rays originating from a basepoint, up to the equivalence relation identifying any two 
rays with bounded Hausdorff distance. The boundary equipped with the compact-open topology. 
It turns out that the visual boundary is not very suitable for sup metrics, as it is easily seen to
have a collapsed topology:  in particular, it is not Hausdorff.  
Indeed, the rays $\alpha^{\rm NE}_m$ and $\alpha^{\rm EN}_m$ are all mutually equivalent to the straight straight northeast geodesic $\alpha^{\rm NE}_0$, but converge to $\alpha^{\rm N}$ and $\alpha^{\rm E}$ respectively as $m\to\infty$.
This means that any open neighborhood of North contains Northwest, and with similar constructions one deduces that the topology on the boundary is trivial.  (See \cite{kitz-rath} for a similar
argument in the visual boundary of $\Z^2$.)

The horofunction boundary is often a better choice for metric spaces outside of the negative curvature setting.  It is defined by using the distance function to embed a metric space 
in the space of continuous functions up to constants, then passing to the topological boundary.  One way to make this precise is to declare that a sequence of points $z_n$
in a metric space $X$ converges to a horofunction $h$ via the formula
$$h(z)=\lim_{n\to\infty} d(z_n,z)-d(z_n,z_0),$$
if that limit exists.  
One easily verifies, via the triangle inequality, that a sequence of points following a geodesic always converges to a horofunction. Such horofunctions are called Busemann functions, or 
Busemann points in the boundary. Denote by $h^*$ and $h^*_m$ the horofunctions which are the limits of the families of geodesics $\alpha^*$ and $\alpha^*_m$, respectively, defined above.
In the other direction, it is not always the case that every horofunction is induced in this way.
In the case of the sup metric, 
 a result of Karlsson, Metz, and Noskov \cite{karlsson-horoballs} tells us that all horofunctions in $\bhor(X)$ are Busemann functions.
We will give a complete description of the boundary in terms of the representative geodesics defined in the last section. We note that our description of the horofunction boundary agrees with that given by Gutierrez in \cite{gutierrezhorofunction}.




\begin{thmd} The horofunction boundary of $(\R^2,\sup)$ is homeomorphic to the circle $S^1$, and parametrized by the normal forms $\alpha^*$ for and $\alpha^*_m$
for $m\ge 0$, as shown in Figure~\ref{fig:horosup}.  \end{thmd}

\begin{proof}

Suppose that $\{z_n = (x_n,y_n)\}$ is a northerly geodesic sequence. For each of the other directions, similar arguments apply. Since $z_n$ is northerly, $y_n \geq |x_n|$ for all $n$, and so there are two possible scenarios. As $n \to \infty$, either (1) the difference $y_n - |x_n|$ is unbounded, or (2) there exists $m \geq 0$ such that $y_n - |x_n| <m$ for all $n$.

\mybox{Case 1:} Suppose $y_n - |x_n|$ is unbounded, and let $z = (x,y)$ be a point in $\R^2$.
If the limit exists, the horofunction corresponding to the sequence $\{(x_n,y_n)\}$ is
\[
h(z) = \lim\limits_{n\to\infty} \sup(|x - x_n|, |y-y_n|) - \sup(|x_n|,|y_n|).
\]
No matter what the input values $x$ and $y$ are, since $y_n - |x_n|$ is unbounded, we get
\[
h(z) = \lim\limits_{n\to\infty}|y-y_n| - |y_n|= -y,
\]
which is the horofunction $h^{\rm N}(z)$ coming from the north geodesic $\alpha^{\rm N}$.

\medskip

\mybox{Case 2:}  Suppose $y_n - |x_n|$ is bounded, and let $m = \sup\{y_n - |x_n|\}=\lim (y_n - |x_n|)$. Since the difference $y_n - |x_n|$ increases monotonically, the sup and limit are equal. Observe that the boundedness implies that for sufficiently large $n$ either $z_n$ is in the first quadrant or the second quadrant. Indeed, going back and forth between quadrants would necessarily increase the value of $m$. Assume $z_n$ eventually stays in the first quadrant, so eventually $y_n - |x_n| = y_n - x_n \to m$.

As $n \to \infty$ both $x_n$ and $y_n$ are growing arbitrarily large, so
$$h(z) = \lim_{n\to\infty}\sup(|x-x_n|, |y - y_n|) - \sup(|x_n|, |y_n|)= \lim_{n\to\infty}\sup(x_n-x, y_n-y) - y_n.$$

If $y \geq x + m$, we see 
$h(z) = \lim\limits_{n\to\infty} x_n - x - y_n = -x -m$.
On the other hand, if $y < x + m$, the sup picks out $y_n - y$, and $h(z) = -y$.
Thus,
$$h(z) = \lim\limits_{n\to\infty} \begin {cases}
-x -m, & y\geq x + m\\
-y, & y <x+m
\end{cases}
= \max (-x-m, -y).$$
This is exactly equal to the horofunction $h^{\rm NE}_m(z)$ coming from the geodesic $\alpha^{\rm NE}_m$.

\medskip

Because the topology is inherited from the compact-open topology on continuous functions $C(X)$,
the $\alpha^*_m$ interpolate between the diagonal and axial directions as $m$ varies from $0$ to $\infty$. 
(Note from the explicit expression for $h^*_m$ that two horofunctions in the same sector with $m$ perturbed by $\epsilon$ gives output that differs by no more than $\epsilon$ on the whole space.) 
Thus the horofunction  boundary is  homeomorphic to $S^1$.
\end{proof}

\begin{figure}[ht]
\begin{tikzpicture}[scale=.5]

\node at (5.4,4.4) [above ] {\fbox{$\max(-x,-y)$}};
\node at (-5.4,4.4) [above ] {\fbox{$\max(x,-y)$}};
\node at (6.8,-3.5) [below ] {\fbox{$\max(-x,y)$}};
\node at (-6.8,-3.5) [below ] {\fbox{$\max(x,y)$}};
\node at (0,5) [above ] {\fbox{$-y$}};
\node at (0,-5) [below ] {\fbox{$y$}};
\node at (5,0) [right ] {\fbox{$-x$}};
\node at (-5,0) [left ] {\fbox{$x$}};
\draw (0,0) circle (4);
\draw [->] (2.2,7) --(1.8,4.1);
\node at (4,7) [above] {\fbox{$\max(-x-m,-y)$}};
\draw [->] (-2.2,7) --(-1.8,4.1);
\node at (-4,7) [above] {\fbox{$\max(x-m,-y)$}};
\draw [->] (5.2,2) --(4.1,1.8);
\node at (5,2) [right] {\fbox{$\max(-x,-y-m)$}};
\draw [->] (5.2,-2) --(4.1,-1.8);
\node at (5,-2) [right] {\fbox{$\max(-x,y-m)$}};
\draw [->] (-5.2,2) --(-4.1,1.8);
\node at (-5,2) [left] {\fbox{$\max(x,-y-m)$}};
\draw [->] (-5.2,-2) --(-4.1,-1.8);
\node at (-5,-2) [left] {\fbox{$\max(x,y-m)$}};
\draw [->] (2.2,-5.4) --(1.8,-4.1);
\node at (4,-7) [above] {\fbox{$\max(-x-m,y)$}};
\draw [->] (-2.2,-5.4) --(-1.8,-4.1);
\node at (-4,-7) [above] {\fbox{$\max(x-m,y)$}};
\foreach \x in {0,90,180,270}
{\begin{scope}[rotate=\x]
\clip (0,0) circle (4.2);
\draw [blue,line width=3,opacity=.5] (0,0)--(4,4);
\draw [blue,line width=3,opacity=.5] (0,.5)--++(4,4);
\draw [blue,line width=3,opacity=.5] (0,1)--++(4,4);
\draw [blue,line width=3,opacity=.5] (.5,0)--++(4,4);
\draw [blue,line width=3,opacity=.5] (1,0)--++(4,4);
\end{scope}}
\draw [blue!50!white,line width=5,opacity=1] (0,-4.2)--(0,4.2);
\draw [blue!50!white,line width=5,opacity=1] (-4.2,0)--(4.2,0);
\draw [<->] (-5,0)--(5,0);
\draw [<->] (0,-5)--(0,5);
\draw (-4.2,-4.2)--(4.2,4.2);
\draw (-4.2,4.2)--(4.2,-4.2);
\end{tikzpicture}
\caption{Every horofunction of $(\R^2,\sup)$ is of one of these types, for some $m> 0$.
Representative  geodesics are shown in blue.}\label{fig:horosup}
\end{figure}
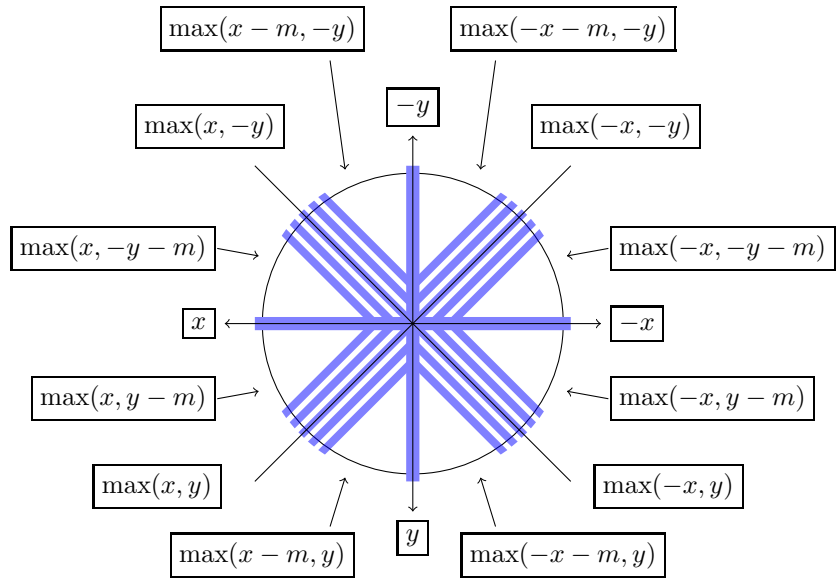


\subsection{Stars in sup metrics}
To understand stars, first consider halfspaces in $(\R^2, \sup)$. 
As shown in Figure~\ref{sup-fig}, the halfspaces separating two points along most straight rays contain half of that circle, while the those separating points in the four axial directions 
contain three-quarters of that circle. We'll see that this phenomenon of ``being able to see more" from axial directions extends to stars.

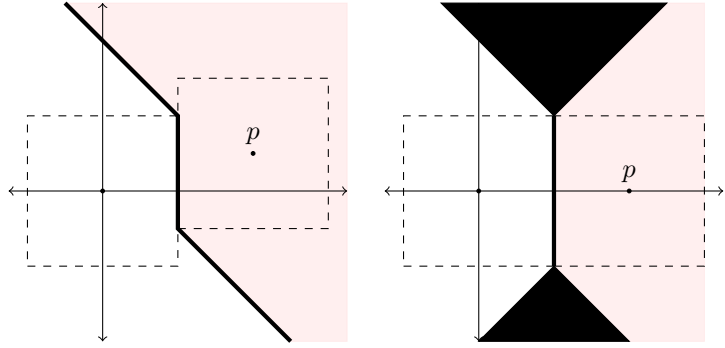
\begin{figure}[H]

\begin{tikzpicture}[scale=.5]
\filldraw [pink!50,opacity=.50] (-1,5)--(2,2)--(2,-1)--(5,-4)--(6.5,-4)--(6.5,5)--cycle;
\draw [<->] (-2.5,0)--(6.5,0);
\draw [<->] (0,-4)--(0,5);
\filldraw (0,0) circle (0.05);
\filldraw (4,1) circle (0.05) node [above] {$p$};
\draw [dashed] (-2,-2) rectangle (2,2);
\draw [dashed] (2,-1) rectangle (6,3);
\draw [ultra thick] (-1,5)--(2,2)--(2,-1)--(5,-4);

\begin{scope}[xshift=10cm]
\filldraw [pink!50,opacity=.50] (2,2)--(2,-2)--(4,-4)--(6,-4)--(6,5)--(5,5)--cycle;
\draw [<->] (-2.5,0)--(6.5,0);
\draw [<->] (0,-4)--(0,5);
\filldraw (0,0) circle (0.05);
\filldraw (4,0) circle (0.05) node [above] {$p$};
\draw [dashed] (-2,-2) rectangle (2,2);
\draw [dashed] (2,-2) rectangle (6,2);
\draw [ultra thick] (2,2)--(2,-2);
\filldraw  (5,5)--(2,2)--(-1,5)--cycle;
\filldraw  (4,-4)--(2,-2)--(0,-4)--cycle;
\end{scope}

\end{tikzpicture}

\caption{In the first figure, the locus of points equidistant from the origin and the point $p$ contains two
rays and a line segment, drawn here with heavy lines.  In the second
case, the locus of points equidistant from the origin and $p=(b,0)$ contains a segment and two infinite cones as pictured.
The dotted squares in both cases are metric spheres centered at $0$ and $p$.  The
strict halfspaces $H(p,0)$ in $(\R^2,\sup)$ are shown in pink.\label{sup-fig}}
\end{figure}

\begin{lemma}\label{nonaxial-lemma}
For nonaxial directions, a sequence $z_n = (x_n, y_n)$ converges to $h^*_m$, $m\geq 0$ if and only if the sequence converges to the geodesic $\alpha^*_m$. 
\end{lemma}

For example, $x_n \to h^{\rm NE}_m$ if and only if $x_n, y_n \to \infty$ and $y_n-x_n \to m$.

\begin{proof}
Without loss of generality, we will examine the $\alpha^{NE}_m$ case.  
Suppose a sequence $z_n= (x_n, y_n)$ converges to  $h^{\rm NE}_m$. Then
\begin{align*}
\lim_{n\to \infty} d(z_n, z) - d(z_n, z_0) &= \lim_{n \to \infty}\sup (|x_n - x|, |y_n -y|) - \sup(|x_n|, |y_n|)\\
& = h^{\rm NE}_m(x)
 = \begin{cases}
-x -m, & y \geq x + m\\
-y, & y< x +m.
\end{cases}
\end{align*}
To obtain these signs for $x$ and $y$ in the limit, we first need $x_n,y_n \to \infty$.
Furthermore,  $-x$ appears in the limit  iff $y \geq x + m$, which occurs iff the first sup in the difference  picks out $|x_n - x|$. 
We have 
$$|x_n -x| = x_n - x \geq y_n - y = |y_n - y|,$$
which can be rewritten $y \geq x + (y_n - x_n)$ and is equivalent to $y \geq x + m$. Therefore, $(y_n - x_n) \to m$. 
\end{proof}

\begin{thme}
The stars of axial boundary points $h^*$ are closed hemispheres, while stars of nonaxial boundary points $h^*_m$ are closed axial quadrants. 
\end{thme}

For example,
$$\Star(h^{\rm E}) = \{h^{\rm N}, h_m^{\rm{NE}}, h_m^{\rm{EN}}, h^{\rm E}, h_m^{\rm ES}, h^{\rm{SE}}_m, h^{\rm S} \mid m \geq 0\}$$
and for any $m\ge 0$,
$$\Star(h^{\rm NE}_{m}) = \{h^{\rm N}, h_\ell^{\rm NE}, h_\ell^{\rm{EN}}, h^{\rm{E}}\mid \ell\geq 0\}.$$

\begin{proof}
Without loss of generality, we will find the stars of $h^{\rm E}$ and then $h^{\rm NE}_m$ 
Let $z_0 = (0,0)$.
First, we  show that for all $m \geq 0$, $h^{\rm EN}_m$ and $h^{\rm NE}_m$ are in $\Star(h^{\rm E})$. Fix the eastern sequence  $z_n = (n,0)$, east-north sequence $w_n = (n, n-m)$, and north-east sequence $w_n'= (n-m, n)$. Then $d(w_n, z_0)~=~d(w_n', z_0)~=~n$. We have
$$d(w_n, z_n) = n-m \leq n ~\quad\hbox{\rm and}~\quad d(w_n', z_n)  = n \leq n.$$
By the sequence criterion, $h^{\rm EN}_m$ and $h^{\rm NE}_m$ are in $\Star(h^{\rm E})$, and since stars are closed, we get 
$h^{\rm N}\in \Star(h^{\rm E})$ by sending $m\to\infty$. Symmetry also gives us the corresponding east-south, south-east, and southern horofunctions.

Now to show that no other boundary points are contained in $\Star(h^{\rm E})$, it suffices to show that $h^{\rm NW}_m \not\in \Star(h^{\rm E})$ for any $m > 0$. Let $z_n = (x_n, y_n)$ be a sequence converging to $h^{\rm E}$. That is, $x_n \to +\infty$ and $x_n - |y_n|$ is unbounded. Define the neighborhood $U = \{ (x, y) \mid y > 2m,  -x +m-1 <y  < -x + m + 1\}$ of $h^{\rm NW}_m$, and suppose $w_n = (u_n, v_n) \to U$. For any $C > 0$, we must show that for $n$ large enough $d(w_n, z_n) > d(w_n, z_0) + C$. Clearly if $w_n$ stays in a bounded set, this is true, so assume $w_n$ leaves all bounded sets. Then the coordinates of $w_n$ satisfy $-u_n+m - 1 < v_n < -u_n + m +1$, and $d(w_n, z_0) < v_n < -u_n + m + 1$. For $n$ large enough, $d(w_n, z_n)  = x_n - u_n$, and as $x_n \to \infty$,
\[
d(w_n, z_n) = -u_n + x_n > -u_n + m + 1 + C > d(w_n, z_0) + C.
\]

Now we find the star of $h_{m}^{\rm NE}$. Consider $z_n = (n- m, n)$ converging to $h_{m}^{\rm NE}$ and sequences $w_n = (n-\ell, n)$ and $w_n' = (n, n-\ell)$ converging to $h_{\ell}^{\rm NE}$ and $h_{\ell}^{\rm EN}$, respectively. Clearly $d(w_n, z_0) = d(w_n',z_0) = n$, and both $d(w_n, z_n)$ and $d(w_n',z_n)$ are finite for all $n$. Thus $\Star(h_{m}^{\rm NE})$ contains the northern, north-east, east-north, and eastern horofunctions.

To show there is nothing else in the star, it suffices to show $h_{\ell}^{\rm NW}$ is not contained in $\Star(h_{m}^{\rm NE})$ for any $\ell >0$.  Let $C >0$.  As above, consider the neighborhood $U = \{ (x, y) \mid y > 2\ell, -x + \ell-1 < y < -x + \ell + 1\}$ of $h_{\ell}^{\rm NW}$, and assume $w_n$ leaves all bounded sets. Again we have $d(w_n, z_0) < v_n < -u_n + m + 1$, and Lemma \ref{nonaxial-lemma} ensures that for large enough $n$, $d(w_n, z_n) = \sup(|x_n - u_n|, |y_n - v_n|) = x_n - u_n$. As $x_n \to \infty$,
\[
d(w_n, z_n) = -u_n + x_n > -u_n + m + 1 + C > d(w_n, z_0) + C.
\]
\end{proof}

In particular, we observe that the star of $h^{\rm E}$ contains $h^{\rm N}$, and vice versa, which is what is needed to make the necessary arguments for \Teich space below.

\subsection{Generalizing to sup metrics in $\R^n$}
It is not difficult to see how these results generalize to $\R^n$. Indeed, the horofunction boundary of $\R^n$ with the sup metric is homeomorphic to $S^{n-1}$. The functions in the boundary can be parametrized in the following way. Let $\epsilon = (\epsilon_1, \ldots, \epsilon_n) \in \{-1,1\}^n$, and let $m = (m_1, \ldots, m_n) \in [0,\infty]^n$, where at least one of the $m_i$ is equal to 0. Then there is a horofunction $h$ in the boundary
\[
h(x_1,\dots,x_n) = \max_i \left\{\epsilon_i  x_i - m_i \right\},
\]
where we use the convention that, for instance,  $\max (x_1 - \infty, -x_2 - 3, x_3)=\max( -x_2 - 3, x_3)$.
The figure below shows the horofunctions reached by sequences which eventually stay in the negative orthant of $\R^3$, corresponding to $\epsilon~=~(+1,+1,+1)$. 
\begin{figure}[ht]
\begin{tikzpicture}[scale=.5]
\draw(-5, 0) -- (5, 0) -- (0, 8.66) -- cycle;
\filldraw (-5,0) circle (0.1);
\filldraw (5,0) circle (0.1);
\filldraw (0,8.66) circle (0.1);
\filldraw (2.5,4.33) circle (0.1);
\filldraw (-2.5,4.33) circle (0.1);
\filldraw (0,0) circle (0.1);
\filldraw (0,2.889) circle (0.1);
\draw [dashed] (0,0) -- (0, 8.66);
\draw [dashed] (2.5,4.33) -- (-5,0);
\draw [dashed] (-2.5,4.33) -- (5,0);

\node at (5,0) [right ] {\fbox{$y$}};
\node at (-5, 0) [left ] {\fbox{$x$}};
\node at (0, 8.66) [above ] {\fbox{$z$}};

\node at (2, 7) [right] {\fbox{$\max(x-\ell, y -m, z), \ell \geq m$}};
\draw [->] (2,7) --(1, 5.5);
\node at (-2, 7) [left] {\fbox{$\max(x-\ell, y -m, z), \ell \leq m$}};
\draw [->] (-2,7) --(-1, 5.5);

\node at (4.5, 2) [right] {\fbox{$\max(x-\ell, y, z-n), \ell \geq n$}};
\draw [->] (4.5,2) --(2.7, 3);
\node at (-4.5, 2) [left] {\fbox{$\max(x, y-m, z-n), m \geq n$}};
\draw [->] (-4.5,2) --(-2.7, 3);

\node at (-2.5, -2) [left] {\fbox{$\max(x, y-m, z-n), m \leq n$}};
\draw [->] (-3,-1.2) --(-1.7, 1);
\node at (2.5, -2) [right] {\fbox{$\max(x-\ell, y, z-n), \ell \leq n$}};
\draw [->] (3,-1.2) --(1.7, 1);

\filldraw[white] (-2.9, 1.4) rectangle (2.9, 2.6);

\node at (2.8, 4.4) [right ] {\fbox{$\max(y,z)$}};
\node at (-2.8, 4.4) [left ] {\fbox{$\max(x,z)$}};
\node at (0,0) [below ] {\fbox{$\max(x,y)$}};
\node at (0,2.889) [below ] {\fbox{$\max(x,y,z)$}};

\end{tikzpicture}
\caption{The simplex in the  horofunction boundary of $\R^3$ with the sup metric. The full boundary is an octahedron with eight such faces.}\label{fig:higher-dim-bdry}
\end{figure}
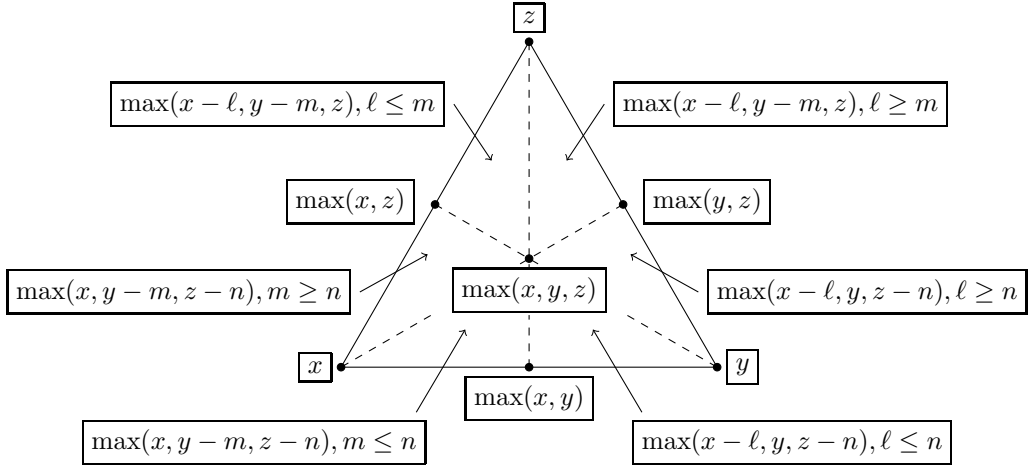

This gives the sphere boundary a simplicial structure, which helps in identifying stars of boundary points. We have parametrized the boundary $\partial_h(\R^n,\sup)$ 
as an orthoplex (i.e., co-cube) with $2^n$ simplex faces.  
If $h$ is a function in the boundary, let $\sigma(h)$ be the simplex of minimal dimension containing $h$. Then the star of $h$ is the simplicial star of $\sigma(h)$, that is, the collection of simplices which have $\sigma(h)$ as a face.
\section{\Teich stars}

We will show that the star-diameter of zero sets  it is at most one; that is,
$Z(F)\subseteq S(F)$.

The crucial element is Minsky's theorem identifying the metric on the thin parts of $\T(S)$:  Minsky finds sup metrics
inside of these thin parts, and that is enough to bound the star-diameter.

\begin{theorem}[Minsky,  \cite{minsky-product}]
Let $\Gamma$ be a set of disjoint curves on $S$.  Let $\epsilon$-$\Thin_\Gamma$ be the subset of $\T(S)$ consisting of points for which the hyperbolic length of every curve in $\Gamma$ is less than $\epsilon$. 
Let $$\epsilon\hbox{-}\Prod\nolimits_\Gamma = \bigl(  \T(S') \times \prod_{\gamma\in \Gamma} \H_\gamma^\epsilon \ , \  \sup \bigr),$$
where $\H_\gamma^\epsilon$ is the metric horoball in $\H^2$ defined by $\{z=x+iy\in \C \mid y>1/\epsilon\}$, thought of as parametrizing the Fenchel-Nielsen length and twist 
coordinates $(\ell,\tau)$ for $\gamma$
via $y=1/\ell$ and $x=\tau$.  Then for sufficiently small $\epsilon$, there is a constant $\cc$ such that
$\epsilon$-$\Thin_\Gamma$ is $(1,\cc)$-quasiisometric to $\epsilon$-$\Prod_\Gamma$.
\end{theorem}

From now on we fix $\epsilon$ to be as short as is required in this theorem, and we write simply $\Thin_\Gamma$ and $\Prod_\Gamma$.

\begin{theorem} If $A$, $B$ are disjoint multicurves, then $B \in \Star(A)$.
\label{multi}\end{theorem}

Let $\Gamma =\{\gamma_1,\ldots,\gamma_n\}$ be the multicurve $A\cup B$.
The orthants of $(\R^n,\sup)$ isometrically embed in $\Prod_\Gamma$,  so they embed
 with only additive distortion in $\Thin_\Gamma$. 
Since  stars
are  large in sup metrics, we will conclude that they are large in \Teich space.  We first write the case of simple closed curves.

\begin{proposition}[The curve case]\label{propSCC}
For disjoint simple closed curves $\alpha$ and $\beta$, $$\beta\in \Star(\alpha).$$
\end{proposition}

\begin{proof}  Let $S'=S_{\alpha,\beta}$ be the subsurface of $S$ obtained by cutting open $S$ along $\alpha$ and $\beta$.
Let $\sigma$ be an arbitrary basepoint in $\T(S')$. 

Let $k$ be a constant chosen large enough that $x_0':=(\sigma,ki, ki)$ is in $\Prod_{\{\alpha,\beta\}}$.
We let $x_n'= (\sigma, e^n i, ki)$ and let $y_n'=(\sigma, ki, e^n i)$.
Then the corresponding points $x_n, y_n$ in $\T(S)$ move through $\Thin_{\{\alpha,\beta\}}$ along paths that pinch $\alpha$ and $\beta$, respectively,
while leaving other Fenchel-Nielsen coordinates fixed.  But then $x_n\to \alpha\in \PMF$, $y_n\to \beta\in \PMF$. 

$$d(x_n',y_n') = d(y_n', x_0') \implies d(x_n,y_n)\le d(y_n,x_0) + 2\cc.$$

Now the sequence criterion (Lemma C) completes the proof.
\end{proof}

\begin{proof}[Proof of Theorem~\ref{multi}]
Similarly let $k$ be large enough that $(\sigma,ki,\ldots,ki)\in\Prod_\Gamma$, and $x_n'$ have $e^ni$ in the factors corresponding to curves from $A$ while
$y_n'$ has $e^ni$ in the (not necessarily distinct) factors corresponding to curves from $B$.  The rest of the proof is the same.
\end{proof}

A  result of Lenzhen and Masur gives a very useful description of zero-sets for minimal foliations.

\begin{theorem}[Lenzhen-Masur, \cite{len-mas}]
Suppose $F$ and $G$ are minimal foliations with $i(F,G)=0$. 
Then there exists a simultaneous approximation by multicurves:
there is a sequence of maximal multicurves $\{P_n\}$ and there exist weight vectors
$a_n,b_n\in \R^k$ such that $a_n\sdot P_n \to F$ in $PMF$, while $b_n\sdot P_n \to G$.
\label{approx}\end{theorem}

The idea of their proof is that minimal foliations describe geodesics whose projections to moduli space diverge
(leave every compact set); thus some curve is very short at every sufficiently large time.  These short curves
decompose $S$ into pairs of subsurfaces dividing the support of $F$ and $G$ with more and more concentration.

The following proposition is then a direct corollary of the Lenzhen-Masur theorem.

\begin{proposition}[Simultaneous approximations]\label{simult}
For any pair of foliations $F,G$ with $i(F,G)=0$, there is a simultaneous approximation by multicurves.
\end{proposition}

\begin{proof}
Let $\Gamma$ be the multicurve which is the union of all the closed curves in $F$ and $G$.  Every leaf of $F$ is either closed or it is minimal on some
subsurface of $S$; let $S_1,\ldots, S_r$ be the collection of all supporting subsurfaces for the non-closed leaves of $F$ and $G$, and let $F_k$ and $G_k$
be the projections of $F$ and $G$, respectively, to $S_k$.  Then for every $k$, the foliations $F_k$ and $G_k$ are either empty or minimal on $S_k$, and
$i(F_k,G_k)=0$.  Then the multicurve approximation to $F$ and $G$ simultaneously is built by the approximating multicurves on $S_k$ which are obtained
by Theorem ~\ref{approx} along with the curves of $\Gamma$, with appropriate weights.
\end{proof}

\begin{proof}[Proof of Theorem A] To show $\ZZ(F) \subseteq \Star(F)$, 
we assume that $G\in \ZZ(F)$ and must prove that $G\in\Star(F)$.  
By Proposition~\ref{simult}, there is a sequence of pants decompositions $\{P_n\}$ which approach
$F$ with one sequence of weights and $G$ with another.  
But we have $b_n\sdot P_n \in \Star(a_n\sdot P_n)$ because they are disjoint (Theorem~\ref{multi}), so by semicontinuity of stars (Lemma~\ref{semicont}),
this tells us that $G\in \Star(F)$.
\end{proof}

\section{Future questions in \Teich geometry}
Let $\mathcal{S}$ be the set of simple closed curves on the surface $S$, and consider the star metric $\ds$ restricted to $\mathcal{S} \subset \PMF$. Recall that $\ds$ is defined combinatorially as the minimal metric such that
\begin{align*}
d_\star(\xi,\eta)=0 &\iff \xi = \eta,\\
d_\star(\xi,\eta)=1& \iff \eta\in \Star(\xi) ~\hbox{or}~ \xi\in \Star(\eta).
\end{align*}
 Our goal in this section is to compare this metric on $\mathcal{S}$ to the metric $\dc$ coming from the curve graph $\mathcal{C}$. Proposition \ref{propSCC} gave us star membership for disjoint curves, which says
 that $\dc(\alpha,\beta)=1 \implies \ds(\alpha,\beta)=1$.  
This  implies that 
$$\dc(\alpha,\beta)\ge \ds(\alpha,\beta)$$
for arbitrary curves, i.e., the identity map from $(\mathcal S,\dc)$ to $(\mathcal S,\ds)$ is Lipschitz.  In this section, we introduce 
some ideas toward the conjecture that $\dc=\ds$ on $\mathcal S$.

The intuition for stars is that boundary points in the same star are ``hard to separate" with half-spaces.   On this intuition, it seems  natural that points at the opposite ends
of a geodesic should be separable, so in disjoint stars.  
A sufficient condition for this kind of separability is a mild hyperbolicity-like condition on geodesics.

\begin{proposition}
Consider a metric space $X$ with bordification $\overline X$ and a geodesic $\gamma$
with endpoints $\gamma^\pm \in \partial X$.  The following are equivalent.
\begin{itemize}
\item[(SG1)] There exists a compact set $K\subset X$ such that for all sequences
$x_n\to\gamma^+$ and $y_n\to \gamma^-$, the segments $\overline{x_n y_n}$ intersect $K$ for sufficiently large $n$.
\item[(SG2)] There exists a compact set $K\subset X$ and open neighborhoods $V$ of $\gamma^+$
and $W$ of $\gamma^-$ such that any geodesic from $W$ to $V$ intersects $K$.
\end{itemize}
\end{proposition}

\begin{proof}
(SG2) $\implies$ (SG1) is clear because the sequences $x_n,y_n$ eventually enter the neighborhoods $V,W$.

Now suppose $\gamma$  does not satisfy (SG2). Then for all neighborhoods $V$ of $\xi$ and $W$ of $\eta$ and for any compact $K$, there are points $x \in V$ and $y \in W$ such that $\overline{xy} \cap K = \emptyset$. Let $\{V_i\}$ be a countable neighborhood basis of $\xi$ and $\{W_i\}$ a countable neighborhood basis of $\eta$, and fix a compact set $K$. For each $i$, we can find points $x_i \in V_i$ and $y_i \in W_i$ such that $\overline{x_iy_i}$ does not intersect $K$. But then we have sequences $x_i \to \xi$ and $y_i\to \eta$ such that $\overline{x_iy_i}$ does not intersect $K$ for any $i$. We chose $K$ arbitrarily, so this is true for all compact $K$. Thus $\gamma$ does not satisfy (SG1).
\end{proof}

A geodesic satisfying these conditions will be called a {\em sticky geodesic} (see Figure~\ref{sticky-geodesics}), because certain long segments stay close to a basepoint in the middle.
Note that with respect to, say, visual boundaries, Euclidean space does not have any sticky geodesics (because parallel geodesics can have the same endpoints without getting close), while all hyperbolic geodesics are sticky.

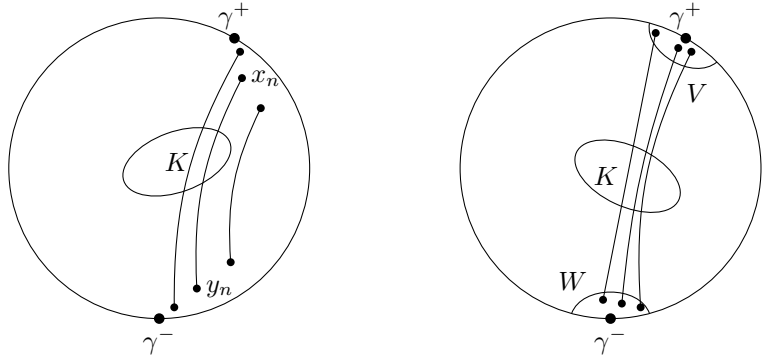
\begin{figure}[H]
\begin{tikzpicture}[scale=.5]
\draw (0,0) circle (4cm);
\fill (0,-4) node[below] {$\gamma^-$} circle (4 pt);
\fill (2,3.46) node[above] {$\gamma^+$} circle (4 pt);
\draw[rotate=20] (.5,0) node {$K$} ellipse (1.5 cm and .8 cm);
\fill (2.7, 1.6) circle (3 pt);
\fill (2.2, 2.4) node[right] {$x_n$} circle (3 pt);
\fill (2.15, 3.1) circle (3 pt);
\fill (1.9,-2.5) circle (3 pt);
\fill (1,-3.2) node[right] {$y_n$} circle (3 pt);
\fill (.4,-3.7) circle (3 pt);
\draw (1.9,-2.5) to[bend left = 15] (2.7,1.6);
\draw (1,-3.2) to[bend left = 15] (2.2, 2.4);
\draw (.4,-3.7) to[bend left = 15] (2.15, 3.1);

\begin{scope}[xshift=12cm]
\draw (0,0) circle (4cm);
\fill (0,-4) node[below] {$\gamma^-$} circle (4 pt);
\node at (-1,-3) {$W$};
\fill (2,3.46) node[above] {$\gamma^+$} circle (4 pt);
\node at (2.3,2) {$V$};
\draw (2.828,2.828) to[bend left = 70] (1.035,3.863);
\draw (-1.035,-3.863) to[bend left = 70] (1.035,-3.863);
\draw[rotate=-25] (.5,0) node[left] {$K$} ellipse (1.5 cm and .8 cm);
\fill (1.2, 3.6) circle (3 pt);
\fill (1.8,3.2) circle (3 pt);
\fill (2.15, 3.1) circle (3 pt);
\fill (-.2,-3.5) circle (3 pt);
\fill (.3,-3.6) circle (3 pt);
\fill (.8,-3.7) circle (3 pt);
\draw (-.2,-3.5) to[bend right = 0] (1.2, 3.6);
\draw (.3,-3.6) to[bend left = 7] (1.8,3.2);
\draw (.8,-3.7) to[bend left = 15] (2.15, 3.1);
\end{scope}
\end{tikzpicture}

\caption{Visualizations of the two definitions of sticky geodesics. \label{sticky-geodesics}}
\end{figure}

\begin{proposition}
If $\gamma$ is a sticky geodesic, then its endpoints are separable by stars:  $\gamma^- \notin \Star(\gamma^+)$
and vice versa.  That is, the endpoints have star-distance at least two.
\end{proposition}
\begin{proof}
By taking the contrapositive of the sequence criterion, we have $\gamma^- \not\in \Star(\gamma^+)$ if and only if there exists a neighborhood $U$ of $\gamma^-$ such that for all sequences $x_n \to \gamma^+$ and $y_n \to U$ and for all $C \geq 0$, we have $d(y_n, x_n) > d(y_n, x_0) + C$ for all $n$ sufficiently large.

Choose neighborhoods $V$ and $W$ of $\gamma^\pm$, respectively, and compact $K$ as in (SG2). Choose any basepoint $x_0 \in K$. Let $x_n \to \gamma^+$ and $y_n \to W$. For $n$ sufficiently large, we know that $d(y_n, x_n) > d(y_n, K) + d(x_n, K)$ because $\overline{x_n y_n}$ hits $K$. We also know that $d(y_n, x_0) < d(y_n, K) + \text{diam}(K)$. Therefore we have that
\[
d(y_n, x_n) - d(y_n, x_0) > d(x_n, K) - \text{diam}(K).
\]
As $n \to \infty$, the right-hand side of this inequality grows larger than any $C \geq 0$. Hence, $\gamma^- \not\in \Star(\gamma^+)$.
\end{proof}

\medskip

Because this is a hyperbolic-like property, it is reasonable to expect for it to hold for thick geodesics and reasonable to hope that it holds for geodesics with ``nice" endpoints.

\begin{conjecture}\label{conj:sticky}
\Teich geodesics with curve endpoints are sticky.
\end{conjecture}

In the Teichm\"uller metric, recall that the condition on which two foliations are joined by a geodesic
is that they {\em jointly fill}, by a result of Gardiner--Masur \cite{Gardiner-Masur}.  Also note that two simple closed
curves jointly fill iff their curve-complex distance is at least three.  (Curves at distance two have a third curve
disjoint from each.)    Thus, if geodesics with curve endpoints are sticky, we can conclude that  
$$\dc(\alpha,\beta) \ge 3 \implies \ds(\alpha,\beta)\ge 2.$$
This would tell us that $\ds(\alpha, \beta) = 1 \implies \dc(\alpha, \beta) \leq 2$, which would establish a $(2,0)$-quasiisometry between the two metrics.

\begin{conjecture} Suppose $\dc(\alpha,\beta)=2$, so that $\alpha$ and $\beta$ jointly fill a proper
subsurface of $S$.
Given sequences $x_n\to \alpha$ and $y_n\to \beta$, the distance $d(x_n,y_n)$ grows faster than 
$d(x_n,x_0)$ for any fixed $x_0$.  
\end{conjecture}

Together,
these would give us the full result identifying the two metrics (at least for basepointed stars):

\setcounter{theorem}{4}
\begin{conjecture}
The star metric $\ds$ and the curve complex distance $\dc$ are isometric on the set of simple closed curves $\mathcal{S}\subset \PMF$.
\end{conjecture}

\bibliography{gen.bib}{}
\bibliographystyle{plain}

\end{document}